\documentclass{article}

\usepackage{arxiv}

\usepackage[utf8]{inputenc} 
\usepackage[T1]{fontenc}    
\usepackage{hyperref}       
\usepackage{url}            
\usepackage{booktabs}       
\usepackage{amsfonts}       
\usepackage{nicefrac}       
\usepackage{microtype}      
\usepackage{lipsum}
\usepackage{amsmath}

\usepackage{amsfonts,amssymb}
\usepackage{graphicx}
\usepackage{amsthm}
\usepackage{xcolor}

\newcommand{\R}{\mathbb{R}}

\newtheorem{theorem}{Theorem}[section]

\newtheorem{lemma}{Lemma}[section]
\newtheorem{corollary}{Corollary}[section]

\newtheorem{remark}{Remark}[section]

\usepackage{cite}
\usepackage{authblk}
\newcommand{\p}{\partial}
\newcommand{\bb}{\begin{equation}}
\newcommand{\ee}{\end{equation}}
\newcommand{\ba}{\begin{array}}
\newcommand{\ea}{\end{array}}
\newcommand{\f}{\frac}
\usepackage[all]{xy}
\newcommand{\ds}{\displaystyle}

\newcommand{\al}{\alpha}
\newcommand{\be}{\beta}

\newcommand{\sign}{\text{sgn}\,}
\numberwithin{equation}{section}

\newcommand{\ad}{\text{ad}}
\newcommand{\Ad}{\text{Ad}}
\newcommand{\arctanh}{\text{arctanh}}

\usepackage{ulem}

\title{Structural and qualitative properties of a geometrically integrable equation}

\author{
Nazime Sales Filho$^{1,2}$\thanks{sales.nazime@gmail.com}\,\,\, and\,\,\, Igor~Leite~Freire$^{3}$\thanks{igor.freire@ufabc.edu.br or igor.leite.freire@gmail.com}\\
$^1$Programa de Pós-Graduação em Matemática, \\
 Universidade Federal do ABC, \\Avenida dos Estados, $5001$, Bairro Bangu,
$09.210-580$,\\
Santo Andr\'e, SP - Brasil\\
$^2$Faculdade de Engenharia,\\Universidade Federal de Mato Grosso,\\Av. Fernando Corrêa da Costa, nº 2367, Bairro Boa Esperança, 78060-900\\ Cuiabá, MT - Brasil\\
$^3$Centro de Matem\'atica, Computa\c{c}\~ao e Cogni\c{c}\~ao\\
Universidade Federal do ABC\\
Santo Andr\'e, Brazil}
 
\begin{document}
\maketitle
\begin{abstract}
Lie symmetries of a Novikov geometrically integrable equation are found and group-invariant solutions are obtained. Local conservation laws up to second order are established as well as their corresponding conserved quantities. Sufficient conditions for the $L^1$ norm of the solutions to be invariant are presented, as well as conditions for the existence of positive solutions. Two demonstrations for unique continuation of solutions are given: one of them is just based on the invariance of the $L^1$ norm of the solutions, whereas the other is based on well-posedness of Cauchy problems. Finally, pseudo-spherical surfaces determined by the solutions of the equation are studied: all invariant solutions that do not lead to pseudo-spherical surfaces are classified and the existence of an analytic metric for a pseudo-spherical surface is proved using conservation of solutions and well-posedness results.
\end{abstract}
{\bf MSC classification 2010:} 35A01, 74G25, 37K40, 35Q51.

\keywords{Conservation laws \and Novikov equation \and Unique Continuation \and Pseudo-spherical surfaces}

\newpage
\section{Introduction}\label{sec1}

In \cite{nov} Novikov classified equations of the type
$$u_t-u_{txx}=F(u,u_x,u_{xx},u_{xxx}),$$
for some function $F$, having an infinite hierarchy of quasi-local higher symmetries. Among the members in his classification are the Camassa-Holm (CH) \cite{chprl}, Degasperis-Procesi (DP) \cite{depro}, and Novikov  \cite{nov,hone2} equations. 

The aforementioned equations are probably the most well known ones falling into Novikov's classification. However, they are not the {\it only} and some of them seem to have very interesting properties shared with them.

In the present work we are interested in the equation
$$
u_t-u_{txx}=\p_x(2-\p_x)(1+\p_x)u^2,
$$
or its equivalent (and more convenient) form
\bb\label{1.0.1}
u_{t}-u_{txx}=4uu_{x}+2u_{x}^{2}+2uu_{xx}-6u_{x}u_{xx}-2uu_{xxx},
\ee
also discovered by Novikov \cite{nov}.

Our attention was attracted to \eqref{1.0.1} due to the fact that it describes pseudo-spherical surfaces (PSS) \cite{rodrigo}, where it was shown that if $u$ is a smooth solution of \eqref{1.0.1}, then
\bb\label{1.0.2}
\omega_1=f_{11}dx+f_{12}dt,\,\,\omega_2=f_{21}dx+f_{22}dt,\,\,\omega_3=f_{31}dx+f_{32}dt,
\ee
where
\bb\label{1.0.3}
\begin{array}{llll}
f_{11}=u-u_{xx}, &f_{12}=\ds{2u(u-u_{xx})+\frac{4}{m_{1}}uu_x-2u_x^{2}-2u^{2}},\\
 f_{21}=\mu (u-u_{xx})\pm m_{1}\sqrt{1+\mu^{2}},&
f_{22}=\ds{\mu\left[2u(u-u_{xx})+\frac{4}{m_{1}}uu_x-2u_x^{2}-2u^{2}\right]},\\
f_{31}=\pm\sqrt{1+\mu^{2}}(u-u_{xx})+m_{1}\mu,&f_{32}=\pm\ds{\sqrt{1+\mu^{2}}\left[2u(u-u_{xx})+\frac{4}{m_{1}}uu_x-2u_x^{2}-2u^{2}\right]},
\end{array}
\ee
$m_1\in\{-2,1\}$, and $\mu\in\R$, satisfy the compatibility conditions
\bb\label{1.0.4}
d\omega_1=\omega_3\wedge\omega_2,\quad d\omega_2=\omega_1\wedge\omega_3,\quad d\omega_3=\omega_1\wedge\omega_2,
\ee
for a PSS with constant Gaussian curvature ${\cal K}=-1$ as long as $\omega_1\wedge\omega_2\neq0$. Moreover, it also implies that \eqref{1.0.1} is the compatibility condition for the zero curvature representation
$$
\p_tX-\p_xT+[X,T]=0,
$$
where $[X,T]=XT-TX$ is the usual commutator between matrices, and 
$$
X=\f{1}{2}\begin{pmatrix}
f_{21}& f_{11}-f_{31}\\
\\
f_{11}+f_{31} & -f_{21}
\end{pmatrix},\quad\text{and}\quad
T=\f{1}{2}\begin{pmatrix}
f_{22}& f_{12}-f_{32}\\
\\
f_{12}+f_{32} & -f_{22}
\end{pmatrix}.
$$

One of the purposes of this work is the investigation of Lie point symmetries \cite{hydon,i1,olverbook} of \eqref{1.0.1}. Having the invariance group, we can then look for invariant solutions and then, using the results from \cite{rodrigo}, we can construct explicit pseudo-spherical surfaces, since
\bb\label{1.0.5}
\ba{lcl}
g&=&\ds{\Big[(u-u_{xx})^2+\Big(\mu (u-u_{xx})\pm m_1\sqrt{1+\mu^2}\Big)^2\Big]dx^2}\\
\\
&+&2\ds{\Big(2u(u-u_{xx})+\psi\Big)\Big[(1+\mu^2)(u-u_{xx}) \pm m_1\mu\sqrt{1+\mu^2}\Big]dxdt}\\
\\
&+&\ds{(1+\mu^2)\Big(2u(u-u_{xx})+\psi\Big)^2dt^2},
\ea
\ee
where $\psi=4(uu_x)/m_1-2u_x^2-2u^2$, $m_1\in\{-2,1\}$, and $\mu\in\R$, is the metric determined by \eqref{1.0.2}--\eqref{1.0.4}.

A natural and related point to be addressed jointly with symmetries is the existence of conservation laws and conserved quantities. In \cite{rodrigo} it was shown a hierarchy of conservation laws for \eqref{1.0.1}. They are essentially divided in two categories: the trivial ones (see \cite[Theorems 4 and 5]{rodrigo}) or the non-local ones (see \cite[Equations (21)--(23)]{rodrigo}. The latter reinforces the fact that the equation is integrable, while the former does not bring any relevant information for the solutions of the equation. In our work we follow a different direction and use techniques related to group analysis \cite{chev-2007,chev-2010-1,chev-2010-2,chev-2014,chev-2017} to establish local conservation laws for \eqref{1.0.1}.

The paper can be summarised as follows.
\begin{itemize}
    \item In section \ref{sec2} we find the Lie point symmetries of the equation, its optimal system and some group invariant solutions;
    \item In section \ref{sec3} we classify all local conservation laws of equation \eqref{1.0.1} up to second order derivatives;
    \item In section \ref{sec4} we use the conservation laws obtained to make an in-depth exploration about qualitative aspects of the solutions of the equation. We show that depending on initial conditions, solutions of \eqref{1.0.1} conserves the $L^1(\R)$ norm, as well as a sort of exponentially weighted $L^1(\R)$ norm. We also prove a unique continuation result for the solutions of the equation.
    \item In section \ref{sec5} we shed light on geometric aspects of the equation.  We show that some of the invariant solutions of the equation correspond to non-generic solutions, that is, they are very peculiar non-trivial solutions of the equation such that \eqref{1.0.5} does not provide a metric for a pseudo-spherical surface. On the other hand, from some invariant solutions corresponding to generic solutions of \eqref{1.0.1} we construct explicit metrics for the corresponding pseudo-spherical surface. In addition, we also prove the existence of solutions leaving their $L^1(\R)-$norm invariant and giving rise to analytical PSS surfaces whose metrics degenerate as $|x|\rightarrow\infty$.
    \item We discuss our results in section \ref{sec6}, whereas our conclusions are given in section \ref{sec7}. 
\end{itemize}

\section{Lie point symmetries and invariant solutions}\label{sec2}

A Lie point symmetry of a given system of differential equations\footnote{An equation here is seen as a system of a single equation.} is a continuous group of transformations mapping solutions of the system into another solutions of the same equations. It is an intrinsic property of the equation and doubtless can be seen as the equations' DNA.

A Lie point symmetry is the flux of a certain differential operator, its generator. This is somewhat a duality of a (local) Lie group of transformations and its corresponding Lie algebra. Therefore, there is a one-to-one correspondence between a symmetry and its generator and, in practical terms, when one is interested in finding symmetries, usually one looks to their generators.

We will not go further in details about symmetries, but we guide the reader for the references \cite{2ndbook,bk,hydon,i1,olverbook}, where different perspectives are presented in wide and deep ways.

For equation \eqref{1.0.1}, a Lie point symmetry is an operator of the form
\bb\label{2.0.1}
X=\tau(t,x,u)\p_t+\xi(t,x,u)\p_x+\eta(t,x,u)\p_u.
\ee

Since \eqref{1.0.1} is a third order equation, we need to find the third order prolongation of \eqref{2.0.1} and, after this, impose the invariance condition. This will give us an over-determined system of linear partial differential equations to be solved for the unknown $\tau$, $\xi$ and $\eta$. For further details, see \cite[Chapter 2]{olverbook}. Solving it we prove our first result:
\begin{theorem}\label{teo2.1}
Let \eqref{2.0.1} be a Lie point symmetry generator of the equation \eqref{1.0.1}. Then it is a linear combination of the operators \bb\label{2.0.2}
X_1=\p_x,\quad X_2=\p_t,\quad X_3=t\p_t-u\p_u.
\ee
\end{theorem}

Currently there are several packages available for finding Lie symmetries of differential equations, {\it e.g.}, see \cite{chev-2007,chev-2010-1,stelios1,stelios2}. We used these facilities to check the result of Theorem \ref{teo2.1} and for this reason we omit the lengthy calculation needed for its demonstration.

Theorem \ref{teo2.1} says that the symmetries of \eqref{1.0.1} are the same as the CH and DP equations (for instance, see \cite{clark}), showing a first common property shared by them.

We now find the optimal systems of sub-algebras {\it à la} Olver \cite[page 203]{olverbook}, see also \cite[chapter 10]{hydon}, that consists essentially in the classification of the orbits of the adjoint representation of the Lie point symmetry generators and then simplify it as much as possible.

\begin{table}[h!]
	\caption{Commutator table of the generators \eqref{2.0.2}. The $(i,j)$-th entry means $[X_i,X_j]$, while $i$ and $j$ indicate line and column, respectively.}\label{tab1}
	\begin{center}
		\begin{tabular}{c|rrc}
		\toprule
		$[X_i,X_j]$	& $X_1$ & $X_2$ & $X_3$ \\
		\midrule
			$X_1$ & 0 &  0 &  0 \\
			$X_2$ & 0  & 0  & $X_2$  \\
			$X_3$ & 0  & -$X_2$ & 0 \\
		\bottomrule
		\end{tabular}
	\end{center}
\end{table}

We recall that if $X$ is a vector field of a Lie algebra $\frak{g}$, then the one-parameter group of transformations generated by $X$ is $\exp(\epsilon X)$ and the corresponding one-parameter group of adjoint transformations is
$$
\ad\, X\Big|_{Y}=\f{d}{d\epsilon}\Big|_{\epsilon=0} \Ad(\exp(\epsilon X))Y,\quad Y\in\frak{g},
$$
see \cite[page 200]{olverbook}. Moreover, we have $\ad\,X\Big|_Y=-[X,Y]$, see \cite[Proposition 3.8]{olverbook}, and
\bb\label{2.0.3}
\Ad (\exp(\epsilon X))Y=\sum_{n=0}^\infty \f{\epsilon^n}{n!}(\ad X)^n(Y)=Y-\epsilon[X,Y]+\f{\epsilon^2}{2}[X,[X,Y]]-\cdots,
\ee
see \cite[page 202]{olverbook}. Then, the orbits of the adjoint representation can be found from \eqref{2.0.3}. Calculating \eqref{2.0.3} for each generator in \eqref{2.0.2} and using Table \ref{tab1}, we obtain the adjoint classification given in Table \ref{tab2}.
\begin{table}[h!]
	\caption{Adjoint representation of the Lie point symmetry generators of \eqref{1.0.1}. The entry $(i,j)$ means line $i$ and column $j$}\label{tab2}
	\begin{center}
		\begin{tabular}{c|rcc}
			\toprule
			$\Ad(\exp(\varepsilon X_i))X_j$	& $X_1$ & $X_2$ & $X_3$ \\
			\midrule
			$X_1$ & $X_1$ &  $X_2$ &  $X_3$ \\
			$X_2$ & $X_1$  & $X_2$  & $X_3-\varepsilon X_2$  \\
			$X_3$ & $X_1$  & $e^{\varepsilon}X_2$ & $X_3$ \\
			\bottomrule
		\end{tabular}
	\end{center}
\end{table}
 
Let $X=a_1X_1+a_2X_2+a_3X_3$ and consider $\Ad(\exp (\varepsilon X_2))X$. Suppose $a_3\neq0$. Without loss of generality, we may assume that $a_3=1$, meaning that $X=a_1X_1+a_2X_2+X_3$. Then
	\begin{equation*}
		\begin{aligned}
		\tilde{X}&= Ad(exp(\varepsilon X_2))X\\
		&=a_1Ad(exp(\varepsilon  X_2))X_1+a_2Ad(exp(\varepsilon X_2))X_2+Ad(exp(\varepsilon X_2))X_3\\
		&=a_1X_1+(a_2-\varepsilon)X_2+X_3.
		\end{aligned}
	\end{equation*}
Therefore, as long as $\varepsilon=a_2$ then $\tilde{X}=a_1X_1+X_3$ and we cannot proceed with further simplification (in the sense of Remark \ref{rem2.1} below). Then the generator $X=a_1X_1+a_2X_2+a_3X_3$ with $a_3\neq0$ is equivalent to $\tilde{X}=\alpha X_1+X_3$, where $\alpha\in\R$ is an arbitrary constant.
	
Let us now assume $a_3=0$. We may assume that $a_2=1$, otherwise it should be $a_2=0$ and $X=X_1$. Table \ref{tab2} tells us that we cannot simplify such generator further (again, see Remark \ref{rem2.1}).

In summary, the optimal systems of one-dimensional sub-algebras is
\begin{equation}\label{2.0.4}
	X_1=\p_x,\quad\alpha X_1+X_3=\al\p_x+t\p_t-u\p_u \quad \textrm{e}\quad cX_1+X_2=c\p_x+\p_t.
\end{equation}

\begin{remark}\label{rem2.1}
We observe that if $\al=0$, then $\al X_1+X_3$ is reduced to $X_3=t\p_t-u\p_u$. However, in case $\al\neq0$, then $(t,x,u)\mapsto(t,|\al|x,u)$ defines a change of coordinates that maps $\al X_1+X_3$ into $(|\al|/\al)X_1+X_3=\sign(\al)X_1+X_3$, since
$$
\p_x\mapsto \f{1}{|\al|}\p_x.
$$
Therefore, we could split the analysis of this case by taking $\al=0$ or $\al=\pm1$. We, however, prefer to leave $\al\in\R$ and proceed a general analysis. 
A similar observation can be done to the linear combination $c\p_x+\p_t$.

\end{remark}

Let us recall that if 
$$
X=\tau\p_t+\xi\p_x+\eta\p_u
$$
is a Lie point symmetry generator for \eqref{1.0.1}, then imposing that solutions of the system (known as {\it characteristic system})
$$
\f{dx}{\xi}=\f{dt}{\tau}=\f{du}{\eta}
$$
are solutions of \eqref{1.0.1}, it then give us solutions of \eqref{1.0.1} invariant under the flux generated by $X$.

\begin{itemize}
    \item {\bf Solutions invariant under the flux generated by $X_1$.} Such a solution is invariant under the translations $(t,x,u)\mapsto(t,x+\epsilon,u)$, meaning that $u(t,x)=\theta(t)$, for some $\theta$. Substituting such a function into \eqref{1.0.1} we conclude that 
    \bb\label{2.0.5}
    u(x,t)=a,
    \ee where $a$ is an arbitrary constant.
    
    \item {\bf Solutions invariant under the flux generated by $\al X_1+X_3$.} If $\al\neq0$, then the solution of the characteristic system gives $u=e^{-\frac{x}{\alpha}}\theta(z)$, where $z=te^{-\frac{x}{\alpha}}$ and $\theta$ is a function to be determined. From this and \eqref{1.0.1} we obtain
    
    \bb\label{2.0.6}
			\left.\begin{aligned}
				&(\alpha^3-4\alpha)\theta' -5\alpha z\theta'' -\alpha z^2\theta''' +(4\alpha^2-4\alpha-8)\theta^2 +(-2\alpha -18)z^2(\theta')^2\\
				&+(4\alpha^2-10\alpha -38)z\theta\theta' +(-2\alpha -18)z^2\theta\theta''-6z^3\theta'\theta''-2z^3\theta\theta'''=0.
			\end{aligned}\right.
	\ee

The above equation is difficult to be solved, but a simple solution can be easily obtained if $\theta$ is a constant and $\al=-1$ or $\al=2$. In this case we have
\bb\label{2.0.7}
u(t,x)=ae^{x},
\ee
if $\al=-1$, and
\bb\label{2.0.8}
u(t,x)=ae^{-x/2},
\ee
for $\al=2$. Above, $a$ denotes an arbitrary constant.

A more acute observation into the structure of \eqref{2.0.6} suggests the ansatz $\theta(z)=az^p$, $a\neq0$. If $p\neq0$ we find that $p=-(\al+1)$ and we have the solution
\bb\label{2.0.9}
u(t,x)=a\f{e^{x}}{t^{\al+1}}.
\ee

Furthermore, for $p=0$ we have the stationary solutions \eqref{2.0.7} and \eqref{2.0.8}.

On the other hand, for $\al=0$ we conclude that 
\bb\label{2.0.10}
u(t,x)=\f{\theta(x)}{t},
\ee
for some function $\theta$. Substituting the function above into \eqref{1.0.1} we obtain the ODE
$$
\theta-\theta''+4\theta\theta'+2(\theta')^2+2\theta\theta''-6\theta'\theta''-2\theta\theta'''=0.
$$

Likewise the previous case, the ODE we need to solve is complicated, but the ansatz $\theta(x)=e^{\be x}$, $\be\in\R$, gives the solution
\bb\label{2.0.11}
u(t,x)=\f{e^{x}}{t}.
\ee
		
\item {\bf Solutions invariant under the flux generated by $cX_1+X_2$.} From the characteristic system we obtain $u=\theta(z)$, where $z=x-ct$. By \eqref{1.0.1} we conclude that $\theta$ satisfies
		
\begin{equation}\label{2.0.12}
	-c\theta'+c\theta'''-4\theta\theta'-2(\theta')^2-2\theta\theta''+6\theta'\theta''+2\theta\theta'''=0.
\end{equation}
Equation \eqref{2.0.12} can be integrated, and as a result we obtain
\begin{equation}\label{2.0.13}
		-c\theta+c\theta''-2\theta^2+2(\theta')^2-2\theta\theta'+2\theta\theta''+C_1=0,
\end{equation}
where $C_1$ is a constant of integration. Multiplying \eqref{2.0.13} by $e^z$ and integrating once more, we get 
\begin{equation}\label{2.0.14}
		-ce^z\theta+ ce^z\theta'-2e^z\theta^2+2e^z\theta\theta'+C_1e^z+C_2=0.
\end{equation}
If we set $C_2=0$ in \eqref{2.0.14}, then we can find the implicit solution
\bb\label{2.0.15}
		z-\dfrac{1}{2}\ln(2\theta^2 +\alpha\theta-C_1)+\dfrac{\alpha}{\sqrt{\alpha^2 +8C_1}}\,\arctanh\left(\dfrac{\alpha+4\theta}{\sqrt{\alpha^2 +8C_1}}\right)+C=0.
\ee
On the other hand, taking $C_1=C_2=0$ into \eqref{2.0.14} we obtain $\theta=ae^{z}$ as a solution, where $a$ is an arbitrary constant, giving
\bb\label{2.0.16}
u(x,t)=a e^{x-ct}
\ee
as a solution for \eqref{1.0.1}.
\end{itemize}

\section{Conservation laws and conserved quantities}\label{sec3}

Conservation laws are frequent partners of symmetries and they are our focus in this section. Before presenting our findings, we would like to give an {\it en passant} view about conservation laws, conserved vectors and conserved quantities of differential equations. Our focus is on equations with two independent variables $(t,x)$, where $t$ is seen as time while $x$ is regarded as space.

We recall that if $u$ depends on $(t,x)$, then $A[u]$ denotes a smooth function depending on $t,\,x,\,u$ and derivatives of $u$ up to some arbitrary, but finite, order. In particular, $A[0]$ means that $u$ and its derivatives considered are all zero. For example, let $f\in C^1(\R)$ be a non-identically vanishing function and consider 
$$A[u]:=e^x\left[f(t)(2u_{x}^{2}-4uu_{x}+2uu_{xx})-f'(t)(u-u_x)\right].$$

Then $A[u]$ depends on derivatives of $u$ up to second order and as long as $u$ and its spatial derivatives vanish, then $A[u]$ also vanishes.

A {\it conservation law} for an equation with two independent variables $t$ and $x$ is a vanishing divergence on the solutions of the equation
$$
D_t C^0+D_x C^1=0,
$$
where $C^0=C^0[u]$, $C^1=C^1[u]$ and $u$ denotes a solution of the equation.  We assume that such dependence is smooth and it vanishes at $[0]$. The pair $(C^0,C^1)$ is called {\it conserved current} whereas the components $C^0$ and $C^1$ are referred as {\it conserved quantity} and {\it conserved flux}, respectively. These terms are very intuitive. In fact, let us assume that for each $t$ the solution exists, the function $x\mapsto u(t,x)$ is defined over $\R$, and both $C^0$ and $C^1$ vanish as long as $u$ and its derivatives vanish. Integrating the divergence above with respect to $x$ over all real line, we get
$$
\f{d}{dt}\int_\R C^0\,dx=-\int_\R D_x C^1\,dx=-C^1\Big|_{-\infty}^\infty=0.
$$

The equation above implies that the 
$$
\int_\R C^0\,dx
$$
is constant along time under the conditions mentioned and is called {\it conserved quantity}. Very often conserved quantities give useful information about the behaviour of the solutions and they reveal physically relevant information about the phenomena described by the equation, if any.

\begin{theorem}\label{teo3.1}
Up to second order, equation \eqref{1.0.1} admits the following conserved currents
\bb\label{3.0.1}
\big(u-u_{xx},2u_x^2-2u^2-2uu_x+2uu_{xx}\big),
\ee
\bb\label{3.0.2}
\big(e^{-2x}(u+2u_x+u_{xx}),2e^{-2x}(u_t-3u_x^2-3uu_x+u_{tx}-3uu_{xx})\big),
\ee
and 
\bb\label{3.0.3}
\big(f(t)e^{x}(u-u_{xx}),e^x\left[f(t)(2u_{x}^{2}-4uu_{x}+2uu_{xx})-f'(t)(u-u_x)\right]\big).
\ee
\end{theorem}

\begin{proof}
Taking the divergence of \eqref{3.0.1}--\eqref{3.0.3} we obtain, respectively,
$$
\Big[u_{t}-u_{txx}-\Big(4uu_{x}+2u_{x}^{2}+2uu_{xx}-6u_{x}u_{xx}-2uu_{xxx}\Big)\Big],
$$
$$
e^{-2x}\Big[u_{t}-u_{txx}-\Big(4uu_{x}+2u_{x}^{2}+2uu_{xx}-6u_{x}u_{xx}-2uu_{xxx}\Big)\Big],
$$
and
$$
f(t)e^{-2x}\Big[u_{t}-u_{txx}-\Big(4uu_{x}+2u_{x}^{2}+2uu_{xx}-6u_{x}u_{xx}-2uu_{xxx}\Big)\Big],
$$
where $f\in C^1(\R)$. When considered on the solutions of \eqref{1.0.1} the expressions above vanish, which shows that \eqref{3.0.1}--\eqref{3.0.3} are conserved currents for the equation.

To prove the uniqueness, we need to show that they are the only conserved vectors for the equation up to second order. Likewise for finding symmetries, obtaining conserved vectors is a very tedious and lengthy process \cite{2ndbook,bk,i1,olverbook}. Fortunately, we have at our disposal several computational facilities for establishing them \cite{chev-2007,chev-2010-1,chev-2010-2,chev-2014,chev-2017} and for this reason we again omit the details since the uniqueness can be checked using such packages.
\end{proof}

We say that two conserved currents $(C^0,C^1)$ and $(\overline{C}^0,\overline{C}^1)$ are equivalent if they differ by a current $(A,B)$ whose divergence vanishes identically. For example, the current \eqref{3.0.1} is equivalent to 
$$
\big(u,2u_x^2-2u^2-2uu_x+2uu_{xx}-u_{tx}\big),
$$
since
$$\big(u,2u_x^2-2u^2-2uu_x+2uu_{xx}-u_{tx}\big)-\big(u-u_{xx},2u_x^2-2u^2-2uu_x+2uu_{xx}\big)=(u_{xx},-u_{tx})$$
and the divergence of the current $(u_{xx},-u_{tx})$ is identically zero.

From the conserved currents we can easily obtain the conserved quantities for the solutions, jointly with its spatial derivatives, decaying to $0$ at infinity.
\begin{theorem}\label{teo3.2}
Suppose that $u$ is a solution of \eqref{1.0.1} such that $u$ and its derivatives up to second order vanish as $|x|\rightarrow\infty$. Then
\bb\label{3.0.4}
{\cal H}_1(t)=\int_\R u(t,x)dx,
\ee
\bb\label{3.0.5}
{\cal H}_2(t)=\int_\R e^{-2x} (u+2u_x+u_{xx})(t,x)dx,
\ee
and
\bb\label{3.0.6}
{\cal H}_3(t)=\int_\R f(t)e^{x} (u-u_{xx})(t,x)dx,
\ee
where $f\in C^1(\R)$.
\end{theorem}

The conservation law (or the conserved current) of a differential {\it equation}, in the sense previously presented, is an intrinsic property of the equation, differently of conserved quantities, which are a property of the {\it solution} of the equation in conjunction with a conserved current of the latter.

Let us illustrate the comment above. If we take $a=1$ in \eqref{2.0.16} and consider \eqref{3.0.1}. Then, we obtain
$$
\big(u-u_{xx},2u_x^2-2u^2-2uu_x+2uu_{xx}\big)\Big|_{u=e^{x-ct}}=(0,0).
$$
Therefore,
$$
\Big[D_t\big(u-u_{xx}\big)+D_x\big(2u_x^2-2u^2-2uu_x+2uu_{xx}\big)\Big]\Big|_{u=e^{x-ct}}=0.
$$

On the other hand, a straightforward calculation shows that
$$
\int_{-\infty}^\infty e^{x-ct}dx=+\infty,
$$
meaning that $u(t,x)=e^{x-ct}$ does not conserve \eqref{3.0.4}.

The situation is even more dramatic if we consider the ``equivalent current'' 
$$
\big(u-u_{xx},2u_x^2-2u^2-2uu_x+2uu_{xx}\big).
$$

The conserved quantity originated from this current is
$$
\overline{{\cal H}}_1(t)=\int_\R(u-u_{xx})dx.
$$

If $u$ and its derivatives vanishes as long as $|x|\rightarrow\infty$, then ${\cal H}_1(t)$ and $\overline{{\cal H}}_1(t)$ are just the same. However, considering $u(t,x)=e^{x-ct}$ we have already seen that ${\cal H}_1(t)=\infty$, while a simple calculation shows that $\overline{{\cal H}}_1(t)=0$! How and why could we obtain so different results from conserved quantities that should be the same?

The answer to the last question is simple, but subtle: these conserved quantities are not the same for $u(t,x)=e^{x-ct}$, because they are not equivalent and, more importantly, none of them are really conserved quantities for the {\it solution} considered. The reason for such a fact is: in order for a solution of \eqref{1.0.1} to have \eqref{3.0.4} as a conserved quantity, it must necessarily vanish at infinity. The solution $u(t,x)=e^{x-ct}$ does not fulfill such condition and this is the root of the observed discrepancy.

\section{Positive solutions and unique continuation}\label{sec4}

Throughout this section we use some aspects of Sobolev spaces and the inverse of the Helmholtz operator $\Lambda^{2}=1-\p_x^2$ that can be found in \cite{taylor} for further details. The convolution of two functions $f$ and $g$ is denoted by $f\ast g$, whereas derivatives with respect to the first and second arguments of a function $u$ will be denoted by $u_t$ or $\p_tu$ and $u_x$ or $\p_xu$, respectively. 

In view of the last discussion after theorem \ref{teo3.2}, henceforth we shall only consider solutions of \eqref{1.0.1} that, jointly with its first and second order $x-$derivatives, decay to $0$ as $|x|\rightarrow\infty$. This will be assumed without further mention throughout this section.

Let $m(t,x)=u(t,x)-u_{xx}(t,x)$ (note that with the conditions above, we have $m\rightarrow0$ as $x\rightarrow\pm\infty$). Then we can express $u$ in terms of $m$ since $u=g\ast m$, $g(x)=e^{-|x|}/2$. Moreover, we can rewrite \eqref{1.0.1} in two other alternative forms, namely,
\bb\label{4.0.1}
m_t+2um-6u_xm-2um_x=2(u-u_x)^2,
\ee
and
\bb\label{4.0.2}
u_t+u^2-2uu_x=G\ast u^2,
\ee
where $G=g+\p_x g$, that is,
\bb\label{4.0.3}
G(x)=(1-\sign{(x)})g(x).
\ee

Through this section we assume that $u$ is a smooth solution of \eqref{1.0.1} decaying fast enough as $|x|\rightarrow\infty$, as well as its derivatives, such that $G\ast u^2$ is well defined and \eqref{4.0.2} makes sense as an alternative form of \eqref{1.0.1}. Also, we assume that $u$ exists on $(0,T)\times\R$, for a certain $T>0$. Note that it is not in question the problem of uniqueness of solutions.

The quantity \eqref{3.0.4} shows that if $u$ is either non-negative or non-positive, then the quantity ${\cal H}_1$ express the invariance of the $L^1(\R)-$norm of the solutions of \eqref{1.0.1}. In a similar way, assuming that $f(t)=1$ and if $m(t,x)>0$ is either non-negative or non-positive, then the quantity \eqref{3.0.6} implies the conservation of $\|m\|_X$, where $X=L^1(\R,e^x)$ is the space of integrable functions with measure $d\mu=e^xdx$. A natural question is: Does \eqref{1.0.1} admits non-negative or non-positive solutions decaying to $0$ as $|x|\rightarrow0$? In this section we give a positive answer, showing that under certain circumstances we have solutions satisfying the condition $u\geq0$.

Let us recall two preliminary results that will be very useful for our purposes.

\begin{lemma}\label{lema4.1}{\tt(\cite[Theorem 1.1]{liu-jde})} If $s>3/2$ and $u_0\in H^s(\R)$ then there exists $T>0$ and a unique solution $u\in C^0([0,T];H^s(\R))$ of the Cauchy problem
\bb\label{4.0.4}
\left\{
\ba{lcl}
u_{t}-u_{txx}=4uu_{x}+2u_{x}^{2}+2uu_{xx}-6u_{x}u_{xx}-2uu_{xxx},\\
\\
u(0,x)=u_0(x).
\ea
\right.
\ee
Moreover, the solution depends continuously on the initial data $u_0$ and we have the estimate
\bb\label{4.0.5}
\|u(t,\cdot)\|_{H^s(\R)}\leq 2\|u_0\|_{H^s(\R)},\quad\text{for}\,\,0\leq T\leq\f{1}{c_s\|u_0\|_{H^s(\R)}},
\ee
where $c_s$ is a positive constant depending on $s$.
\end{lemma}

\begin{lemma}{\tt(\cite[Theorem 1.4]{liu-jde})}\label{lema4.2}
If the initial data $u_0$ is analytic on $\R$, then there exists an $\epsilon>0$ and a unique solution $u(t,x)$ of the problem \eqref{4.0.4} that is analytic in both variables for $(t,x)\in(-\epsilon,\epsilon)\times\R$. 
\end{lemma}

Let us denote by $u_0$ and $m_0$ the functions $u(0,x)$ and $m(0,x)$, respectively.

\begin{theorem}\label{teo4.1}{\tt (Existence of positive solutions)} Equation \eqref{1.0.1} admits solutions $u\geq0$ with corresponding $m(t,x)$ non-negative. Moreover, $u$ and its derivatives decay to $0$ as $|x|\rightarrow\infty$.
\end{theorem}

\begin{proof}
In view of lemmas \ref{lema4.1} and \ref{lema4.2} we can choose a suitable initial data $u_0\in H^s(\R)$ and $s>3/2$ such that the problem
\bb\label{4.0.6}
\left\{\ba{l}
q_t(t,x)=-2u(t,q(t,x)),\\
\\
q(0,x)=x,
\ea
\right.
\ee
for each fixed $x$ has a unique solution. Also, note that the corresponding solution $u$ vanish at infinity, as well as its derivatives.

Differentiating \eqref{4.0.6} with respect to $x$ we obtain
$$
\left\{
\ba{l}
q_{tx}(t,x)=-2u_x(t,q(t,x))q_x(t,x),\\
\\
q_x(0,x)=1.
\ea
\right.
$$

Again, choosing a suitable $s$ and initial data $u_0\in H^s(\R)$, we can proceed similarly as in \cite[Theorem 3.1]{const} to conclude that the function $q(t,\cdot)$ is an increasing diffeomorphism on the line and
$$
q_x(t,x)=\exp{\Big(-2\int_a^tu_x(s,q(s,x))ds\Big)}>0.
$$

Therefore,
$$
\f{d}{dt}m(t,q)=m_t+q_tm_x=m_t-2um_x=-2(u-3u_x)m+2(u-u_x)^2,
$$
that is
$$
\f{d}{dt}\Big(\exp{\Big(2\int_0^t(u-3u_x)(s,q(s,x))ds}\Big)m(t,q)\Big)\geq0,
$$
which yields
\bb\label{4.0.7}
m(t,q(t,x))\geq m_0(x)\exp{\Big(-2\int_0^t(u-3u_x)(s,q(s,x))ds\Big)}.
\ee

Therefore, as long as $m_0(x)\geq0$ we also have $m(t,q(t,x))\geq0$. Since $q(t,\cdot)$ is a diffeomorphism, if $m_0$ is non-negative, then $m(t,\cdot)\geq0$. Finally, since $u=g\ast m$, if $m\geq0$ then $u$ is also non-negative.
\end{proof}

As a consequence of the theorem above, we know that \eqref{1.0.1} has non-negative solutions decaying to $0$ as $|x|\rightarrow\infty$, so that for these solutions \eqref{3.0.4} is equivalent to the conservation of the $L^1(\R)$ norm of $u$. Also, if $e^{x}m$ is integrable and $m$ is non-negative, then \eqref{3.0.6} implies that $e^xm\in L^1(\R)$, or that $m\in L^1(\R;e^x)$, meaning that $m$ belongs to $L^1(\R)$ with measure $d\mu=e^x dx$.

Let $u$ be a solution of \eqref{1.0.1}, $G$ as in \eqref{4.0.3}, and 
\bb\label{4.0.8}
F_t(x)=(G\ast u^2)(t,x)=\int_{-\infty}^{+\infty}G(x-y)u(t,y)^2dy.
\ee

If $t\in(0,T)$ and $a<b$ are real numbers, we then have the representation
\bb\label{4.0.9}
F_t(b)-F_t(a)=\int_{-\infty}^{+\infty}S_{a,b}(y)u(t,y)^2 dy,
\ee
where
\bb\label{4.0.10}
S_{a,b}(y)=G(b-y)-G(a-y)=\f{1-\sign{(b-y)}}{2}e^{-|b-y|}-\f{1-\sign{(a-y)}}{2}e^{-|a-y|}.
\ee

We observe that if $y<a$, then $\sign{(b-y)}=\sign{(a-y)}=+1$ and $S_{a,b}(y)=0$, whereas if $y>b$ $S_{a,b}(y)=e^{b-y}-e^{a-y}=e^{a-y}(e^{b-a}-1)>0$. This proves the following result.
\begin{theorem}\label{teo4.2}
Let $a$ and $b$ two real numbers with $a<b$ and $S_{a,b}$ given by \eqref{4.0.10}. Then $S_{a,b}$ is a real valued function defined over $\R$ and belongs to $L^1(\R)$. Moreover, $S_{a,b}$ is non-negative outside the compact set $[a,b]$.
\end{theorem}

Henceforth we assume that $a$ and $b$ are distinct numbers, with $a<b$.

\begin{theorem}\label{teo4.3}
Suppose that for a given $t_0\in(0,T)$ there exists $[a,b]\subseteq\R$ such that $u_t(t_0,a)=u_t(t_0,b)=0$ and $u(t_0,x)=0$, $x\in[a,b]$. Then $u(t_0,x)=0$ for $x\geq a$.
\end{theorem}

\begin{proof} It suffices to show that $u(t_0,x)=0$ for $x>b$.

From \eqref{4.0.8} and \eqref{4.0.2} we conclude that
$$
F_t(x)=(u_t+u^2-2uu_x)(t,x).
$$
Under the conditions in the theorem, we have $F_{t_0}(a)=F_{t_0}(b)=0$. From the representation \eqref{4.0.9} we have
\bb\label{4.0.11}
\int_{-\infty}^{+\infty}S_{a,b}(y)u(t_0,y)^2dy=F_{t_0}(b)-F_{t_0}(a)=0.
\ee
Since $S_{a,b}(y)u(t_0,y)=0$ for $x\leq b$, and $S_{a,b}(y)>0$ for $y>b$, condition \eqref{4.0.11} implies that $u(t_0,x)=0$ for $x>b$.
\end{proof}

Since we know that \eqref{1.0.1} admits solutions conserving the $L^1(R)$ norm, we can use our precedent result to establish the following unique continuation result.

\begin{corollary}\label{cor4.1}
If $u$ is a solution for \eqref{1.0.1} such that $\|u(t,\cdot)\|_{L^1(\R)}$ is invariant and there exists an open set $\Omega=(T_0,T_1)\times(-\infty,b)\subseteq(0,T)\times\R$, for some $b\in\R$, such that $u$ restricted to $\Omega$ vanishes identically. Then $u\equiv0$.
\end{corollary}

\begin{proof}
Let $t_0\in(T_0,T_1)$ and $a\in\R$ such that $a<b$. Then the solution $u$ and the set $\{t_0\}\times[a,b]$ satisfies the conditions in Theorem \ref{teo4.3}. As a result we know that $u(t_0,x)=0$ for all $x\geq a$, which implies that $u(t_0,x)=0$, $x\in\R$. The invariance of the $L^1(\R)$ norm implies that $u(t,x)=0$ for all $(t,x)\in(0,T)\times\R$.
\end{proof}

We observe that the result in Corollary \ref{cor4.1} does not require the uniqueness of the solution, but imposes the conservation of the $L^1(\R)$ norm of $u$. Corollary \ref{cor4.1}'s demonstration shows that only one solution of \eqref{1.0.1} can satisfy its condition, meaning that it can also be seen as a uniqueness result. We can change its perspective and abandon the hypothesis of the conservation of the $L^1(\R)$ norm (which can be interpreted as mass conservation) and prove an alternative unique continuation result.

\begin{theorem}\label{teo4.4}
Assume that $u\in H^s(\R)$, for $s$ sufficiently large, is a solution of \eqref{1.0.1} vanishing on an open set $\Omega$ as in corollary \ref{cor4.1}. Then $u\equiv0$.
\end{theorem}

\begin{proof}
Choosing $t_0\in(T_0,T_1)$ and using theorem \ref{teo4.3} once again, we conclude that $u(t_0,x)=0$ for $x\geq a$, for some (in fact, all) $a<b$. This implies that $u(t_0,x)=0$. Let us assume $s>3/2$ and consider $v(t,x):=u(t+t_0,x)$. By theorem \ref{teo2.1} $v$ is also a solution of \eqref{1.0.1} subject to $v(0,x)=u(t_0,x)=0$. By lemma \ref{lema4.1} we have a unique solution defined for any $t\in[0,T^\ast)$, where $T^\ast=T-t_0$. 

By \eqref{4.0.5}, for any $t\in(0,T^\ast)$ we have
$$
\|v(t,\cdot)\|_{H^s(\R)}\leq2\|v(0,\cdot)\|_{H^s(\R)}=2\|u(t_0,\cdot)\|_{H^s(\R)}=0.
$$

Moreover, using \eqref{4.0.5} once more, we conclude that $T^\ast=\infty$ and, therefore, we have a unique (global) solution.
\end{proof}

\section{Pseudo-spherical surfaces}\label{sec5}

A two-dimensional manifold ${\cal M}$ is called a pseudo-spherical surface (PSS) with constant Gaussian curvature ${\cal K}=-1$ if there exists a triad of one-forms $\omega_1,\,\omega_2,\,\omega_3$ on ${\cal M}$ satisfying \eqref{1.0.4}, which are structure equations endowing ${\cal M}$ with the metric $g=\omega_1^2+\omega_2^2$. For further details, see \cite{chern,rodrigo,reyes2011}.

A differential equation with two independent variables is said to be of a pseudo-spherical type, or describe a pseudo-spherical surface, if there exists one-forms \eqref{1.0.2}, whose coefficients are functions depending on the independent and dependent variables, and derivatives of the latter, such that \eqref{1.0.4} is satisfied and $\omega_1\wedge\omega_2\neq0$. If so, the domain of the solution, endowed with \eqref{1.0.5}, is a PSS.  

Recently, Freire and Tito showed that \eqref{1.0.1} describes pseudo-spherical surfaces. More precisely, they proved that \eqref{1.0.2}--\eqref{1.0.3} satisfy \eqref{1.0.4} on the solutions of \eqref{1.0.1}. Solutions of an equation describing pseudo-spherical surfaces satisfying $\omega_1\wedge\omega_2\neq0$ are called {\it generic} solutions.

\begin{remark}\label{rem5.1} In \cite{rodrigo} it was shown that the only non-generic solutions of \eqref{1.0.1} are:
\begin{itemize}
    \item For $m_1=-2$, then $\phi(t,x)=\pm\sqrt{ae^{-x}+b}$;
    \item For $m_1=1$, then $\phi(t,x)=\pm\sqrt{ae^{2x}+b}$ or $\phi(t,x)=f(t)e^{x}$.
\end{itemize}
Above, $a$ and $b$ are arbitrary constants, whereas $f(t)$ is a smooth real valued function.
\end{remark}

We now classify the non-generic solutions that are invariant under the fluxes generated by \eqref{2.0.4}. To this end, we proceed as follows: let $X$ denote a vector field \eqref{2.0.4}. We then impose that the non-generic solutions above satisfy the condition
$$
X(u-\phi)\Big|_{u=\phi}=0.
$$

For convenience, we shall only consider the sign $+$ in the solutions above. The sign $-$ is treated similarly and gives the same results. In what follows, $a$, $b$ and $c_1$ denote constants, whereas $f(\cdot)$ is a smooth function. Moreover, we assume that the solution is not identically zero.

\begin{enumerate}
    \item Let us consider $m_1=-2$ and $\phi=\sqrt{ae^{-x}+b}$. Then:
    \begin{enumerate}
        \item The condition $X_1(u-\phi)\Big|_{u=\phi}=0$ implies that $a=0$ and then $u=c_1$.
        
        \item The condition $(\al X_1+X_3)(u-\phi)\Big|_{u=\phi}=0$ gives
        $$
        (\al-2)ae^{-x}-2b=0,
        $$
       whose solution is \eqref{2.0.8}.
       
       \item The condition $(c X_1+X_2)(u-\phi)\Big|_{u=\phi}=0$ implies $a=0$ or $c=0$. If $a=0$ then the solution is constant, whereas whenever $c=0$ we obtain $u(t,x)=\sqrt{ae^{-x}+b}$.
    \end{enumerate}
    
    \item Let us consider $m_1=1$ and $\phi=\sqrt{ae^{2x}+b}$. Then:
    \begin{enumerate}
        \item The condition $X_1(u-\phi)\Big|_{u=\phi}=0$ implies that $a=0$ and then $u$ is a constant.
        
        \item The condition $(\al X_1+X_3)(u-\phi)\Big|_{u=\phi}=0$ gives
        $$
        (\al-1)ae^{2x}-b=0,
        $$
       that is, $(\al-1)a=0$ and $b=0$. A non-trivial solution can only be obtained when $\al=1$, and in this case we obtain \eqref{2.0.7}.
        
       \item The condition $(c X_1+X_2)(u-\phi)\Big|_{u=\phi}$ again implies $a=0$ or $c=0$. If $a=0$ have constants as solutions, whereas $c=0$ gives $u(t,x)=\sqrt{ae^{2x}+b}$.
    \end{enumerate}
    
    \item Let us consider $m_1=1$ and $\phi=f(t)e^{x}$. Then:
    \begin{enumerate}
        \item The condition $X_1(u-\phi)\Big|_{u=\phi}=0$ implies $f(t)=0$, which is a contradiction. This means that we do not have any  non-generic solution of the form $\phi=f(t)e^{x}$ invariant under translations $(t,x,u)\mapsto(t,x+\epsilon,u)$.
        
        \item The condition $(\al X_1+X_3)(u-\phi)\Big|_{u=\phi}=0$ gives
        $$
        f(t)=\f{c_1}{t^{\al+1}}
        $$
        and we then recover \eqref{2.0.9}.
       
       \item The condition $(c X_1+X_2)(u-\phi)\Big|_{u=\phi}=0$ says that $f(t)=ae^{-ct}$ and we re-obtain \eqref{2.0.16}.
    \end{enumerate}
\end{enumerate}
The solutions above are the only invariant solutions (considering the fluxes generated by the fields \eqref{2.0.4}) that do not endow the domains of the solutions of \eqref{1.0.1} with a metric with Gaussian curvature ${\cal K}=-1$.

Let us use the invariant solutions found before to construct explicit metrics for PSS surfaces. For example, if $m_1=-2$ and considering $u(t,x)=e^x/t$, we have the forms
\bb\label{5.0.1}
\omega_1=-6\f{e^{2x}}{t^2}dt,\,\,\omega_2=\mp 2\sqrt{1+\mu^2}dx-6\mu\f{e^{2x}}{t^2}dt,\,\,\omega_3=-2\mu dx\mp 6\sqrt{1+\mu^2}\f{e^{2x}}{t^2}dt,
\ee
which yields the metric
\bb\label{5.0.2}
g=4(1+\mu^2)dx^2\pm 24\mu\sqrt{1+\mu^2}\f{e^{2x}}{t^2}dxdt+36(1+\mu^2)\f{e^{4x}}{t^4}dt^2,
\ee
and a simple calculation shows that
$$
\omega_1\wedge\omega_2=\mp 12\sqrt{1+\mu^2}\f{e^{2x}}{t^2}dx\wedge dt=d\omega_3,
$$
meaning that \eqref{5.0.1} really gives a PSS with metric \eqref{5.0.2}. 

On the other hand, the case $m_1=1$ and the stationary solution $u(t,x)=a e^{-x/2}$, $a\neq0$, provide the forms
\bb\label{5.0.3}
\left.\begin{aligned}
\omega_1=\f{3}{4}ae^{-x/2}dx-3a^2e^{-x}dt,\,\,\omega_2=\Big(\f{3}{4}a\mu e^{-x/2}\pm \sqrt{1+\mu^2}\Big)dx-3a^2\mu e^{-x}dt,\\
\omega_3=\Big(\pm\f{3}{4}a\sqrt{1+\mu^2}e^{-x/2}+\mu\Big) dx\mp3a^2\sqrt{1+\mu^2}e^{-x}dt,
\end{aligned}\right.
\ee
that satisfies
$$\omega_1\wedge\omega_2=\pm 3a^2\sqrt{1+\mu^2}e^{-x}dx\wedge dt=d\omega_3,$$
and gives the metric
\bb\label{5.0.4}
\left.\begin{aligned}
   g&=\Big[\Big(\f{9}{16}a^2e^{-x}+1\Big)(1+\mu^2)\pm\f{3}{2}a\mu\sqrt{1+\mu^2}e^{-x/2}\Big]dx^2\\
   &-\Big[\f{9}{2}a^3(1+\mu^2)e^{-3x/2}\pm 6a^2\mu\sqrt{1+\mu^2}e^{-x}\Big] dxdt+9a^4(1+\mu^2)e^{-2x}dt^2, 
\end{aligned}\right.
\ee

We observe that the one-parameter family of metrics \eqref{5.0.2} and \eqref{5.0.4} are analytic on certain domains of $\R^2$. However, none of them come from conservative solutions. More generally, observe that only non-generic solutions of \eqref{1.0.1} are those characterised by Remark \ref{rem5.1}, and none of them belong to $L^1(\R)$. Lemma \ref{lema4.2} combined with theorem \ref{teo4.1} are key ingredients for the next existence result.

\begin{theorem}{\tt (Existence of analytic pseudo-spherical surface)}
Let $u_0\in H^s(\R)$ be an analytic data such that $m_0=u_0-u_0''>0$. Then, for $s$ sufficiently large, there exists $\epsilon>0$ such that the solutions of \eqref{1.0.1} endows the strip ${\cal S}=(-\epsilon,\epsilon)\times\R$ with the metric giving by \eqref{1.0.5}, with the following properties:
\begin{itemize}
    \item $g_{ij}$ are analytic, $1\leq i,j\leq2$;
    \item $({\cal S},g)$ is a surface with constant curvature;
    \item $\lim\limits_{x\rightarrow\pm\infty}g_{11}>0$, and $g_{i2}\rightarrow0$ as $|x|\rightarrow\infty$.
\end{itemize}
Furthermore, $u\neq0$ everywhere on ${\cal S}$ and conserves the $L^1(\R)$ norm.
\end{theorem}

\begin{proof} We first note that if $u,\,u_x,\,u_{xx}\rightarrow0$ as $|x|\rightarrow\infty$, then $g_{11}\rightarrow m_1^2(1+\mu^2)>0$, and $g_{12},g_{22}\rightarrow0$. 

Let $s>3/2$ large enough, so that theorem \ref{teo4.1} and lemma \ref{lema4.2} hold. By lemma \ref{lema4.2} we conclude the existence of $\epsilon>0$ so that $u\big|_{\cal S}$ is analytic. Due to this fact, from \eqref{1.0.5} we conclude that the metric coefficients are all analytic. Moreover, since $m_0>0$, \eqref{4.0.7} implies that $m(t,q(t,x))>0$ for all $t\in(-\epsilon,\epsilon)$ and $x\in\R$. Taking into account that $q(t,\cdot)$ is a diffeomorphism, we conclude that $m(\cdot,q(\cdot,\cdot))\big|_{{\cal S}}$ does not vanish and $u=g\ast m$ either. Moreover, $u$ cannot be any of the solutions in Remark \ref{rem5.1}, is strictly positive on ${\cal S}$, and conserves the $L^1(\R)$ norm in view of theorem \ref{teo3.2}. 

Since \eqref{1.0.1} is of pseudo-spherical type, we can endow ${\cal S}$ with the metric \eqref{1.0.5}. As a consequence, $({\cal S},g)$ has Gaussian curvature ${\cal K}=-1$ at every point.
\end{proof}

\section{Discussion}\label{sec6}

As previously mentioned, in \cite{nov} it was carried out a classification of symmetry integrable equations of the type $u_t-u_{txx}=F$, where $F$ is a polynomial in $u$ and its spatial derivatives. Equation \eqref{1.0.1} is one of the equations obtained by Novikov, among other more famous and well studied like the Camassa-Holm, Degasperis-Procesi and Novikov equations.

In our work we studied the equation \eqref{1.0.1} mainly from the point of view of group analysis \cite{2ndbook,hydon,i1,olverbook}, that is, we found its invariance group and invariant solutions in section \ref{sec2} and conservation laws and conserved quantities in section \ref{sec3}.

To the best of our knowledge, the only equations discovered by Novikov whose Lie point symmetries are known are just the equations mentioned above: for the CH and DP equations, they can be found in \cite{clark}, while for the Novikov equation they are classified in \cite{boz}, see also \cite{anco,igor-cm}. Note that the work by Clarkson precedes the one by Novikov as well as the discovery of the DP equation \cite{depro}. The only equation properly discovered by Novikov with known symmetries is just the equation named after him, and \eqref{1.0.1}, due to the present work.

Starting from the Lie symmetries, we went further and studied different topics using them as a start point. Namely, we found some explicit $C^\infty$ solutions for the equation in section \ref{sec2} and showed which of them give rise to PSS.

Still related with symmetries, we also established conservation laws for \eqref{1.0.1}. The conserved current \eqref{3.0.1} can be easily inferred from the fact that \eqref{1.0.1} itself is a conservation law. In \cite{rodrigo} it was shown infinitely many conserved currents for the equation. However, most of them are trivial or non-local ones, which in practical terms does not bring much information for the equation, despite the fact that they reinforces the integrability of \eqref{1.0.1}.

The conserved currents \eqref{3.0.2} and \eqref{3.0.3} are very new. They show that for solutions conserving \eqref{3.0.5}, then necessarily $u+2u_x+u_{xx}$ decreases faster then $o(e^{2x})$ as $x\rightarrow-\infty$, whereas those conserving \eqref{3.0.6} imply that $m(t,x)$ decays to $0$ higher than $o(e^{-x})$ as $x\rightarrow+\infty$.

From the point of view of qualitative properties, we prove the existence of positive solutions for \eqref{1.0.1}. A similar result was previously proved in \cite[Section 5]{li-na}, but our demonstration is fairly different. Compare the demonstration of theorem \ref{teo4.1} with \cite[Lemma 5.6]{li-na}. 

It is worth mentioning that we do not need our theorem \ref{teo4.4} to assure the existence of positive solutions for \eqref{1.0.1}. This can easily be inferred from the invariant solutions. For example, consider \eqref{2.0.7}, \eqref{2.0.8} or \eqref{2.0.16} with $a>0$, or \eqref{2.0.11} for $t>0$. These solutions, however, do not have \eqref{3.0.4} as a conserved quantity. The relevance of theorem \ref{teo4.1} relies on the fact that not only it implies the existence of spatially bounded solutions, but also $L^1(\R)$ conservative ones.

For non-negative solutions $u(t,x)$ of \eqref{1.0.1} leaving $t\mapsto \|u(t,\cdot)\|_{L^1(\R)}$ invariant we can prove unique continuation properties as in corollary \ref{cor4.1}. This can be seen as an energy conservation functional approach, as suggested in \cite{igor-jpa,freire-cor}, which is more related to physical properties of the equation, see the discussions presented in \cite{igor-jpa}.

In theorem \ref{teo4.4} we arrive at the same conclusion of corollary \ref{cor4.1} replacing the conservation of the $L^1(\R)$ norm (which can be interpreted as a mass conservation) by the uniqueness of the solution of a Cauchy problem \eqref{4.0.4} and the continuity of the data to solution map as well, as proposed in \cite{linares}.

In regard to geometric aspects of \eqref{1.0.1}, in \cite{rodrigo} it was shown that generic solutions of \eqref{1.0.1} give rise to pseudo-spherical surfaces with Gaussian curvature ${\cal K}=-1$. Moreover, the authors classified all non-generic solutions. Our paper complements the aforementioned reference by showing the invariant solutions that are non-generic and the existence of an analytical pseudo-spherical surface. In \cite[Remark 2]{rodrigo} it was mentioned by {\it en passant} that \eqref{1.0.1} would describe analytical PSS, but such a fact was not proved nor better characterised. Our paper not only rigorously prove the existence of such PSS, but also gives a complete description of the behaviour of the metric for large values of $|x|$, at least for solutions arising from the conditions in lemma \ref{lema4.2} (see \cite[Theorem 1.3]{liu-jde}). Our demonstration, in particular, only guarantees the existence of an analytical PSS for solutions conserving the $L^1(\R)$ norm.

\section{Conclusion}\label{sec7}

In this work we found the Lie point symmetries of the equation \eqref{1.0.1}. From them we classified their optimal system of algebras and obtained explicit group-invariant solutions, as shown in section \ref{sec2}. We established three conservation laws for \eqref{1.0.1}, given in \ref{sec3}, as well as their corresponding conserved quantities. Moreover, we proved the existence of positive solutions for \eqref{1.0.1} as well as unique continuation results for its solutions. Finally, we found explicit pseudo-spherical surfaces determined by the solutions of \eqref{1.0.1} and proved the existence of analytical surfaces with constant Gaussian curvature ${\cal K}=-1$ arising from positive $L^1(\R)$ conservative solutions of \eqref{1.0.1}.

\section*{Data Availability Statement}
The authors declare that data sharing is not applicable to this article as no data sets were generated or analysed during the current study.

\section*{Conflict of Interest Statement}
The authors declare that there is no conflict of interest.

\section*{Acknowledgements}

The first author is thankful to UFMT for the leaving to develop his PhD at UFABC. I. L. Freire thanks FAPESP for financial support (grant nº 2020/02055-0).


\begin{thebibliography}{10}

\bibitem{anco} Anco SC, da Silva PL, Freire IL. A family of wave-breaking equations generalizing the Camassa--Holm and Novikov equations. {\it J Math Phys}. 2015;56:091506.

\bibitem{boz} Bozhkov Y, Freire IL, Ibragimov N. Group analysis of the Novikov equation. {\it Comp Appl Math}. 2014; 33:193--202.

\bibitem{2ndbook} Bluman G, Cheviakov A, Anco SC. {\it Applications of Symmetry Methods to Partial Differential Equations}. Springer; 2010.

\bibitem{bk} Bluman GW, Kumei S. {\it Symmetries and Differential Equations}. Springer; 1989.

\bibitem{chprl} Camassa R, Holm DD. An integrable shallow water equation with peaked solitons. {\it Phys Rev Lett}. 1993;71:1661--1664.

\bibitem{chern} Chern SS, Tenenblat, K. Pseudo-spherical surfaces and evolution equations. {\it Stud Appl Math}. 1986;74, 55--83.

\bibitem{chev-2007} Cheviakov A. GeM software package for computation of symmetries and conservation laws of differential equations. {\it Comp Phys Comm}. 2007;176: 48--61.

\bibitem{chev-2010-1} Cheviakov A. Symbolic computation of local symmetries of nonlinear and linear partial and ordinary differential equations. {\it Math Comput Sci}. 2010;4: 203--222.

\bibitem{chev-2010-2} Cheviakov A. Computation of fluxes of conservation laws. {\it J Eng Math}. 2010;66:153--173.

\bibitem{chev-2014} Cheviakov A. {\it Symbolic computation of nonlocal symmetries and nonlocal conservation laws of partial differential equations using the GeM package for Maple, Similarity and Symmetry Methods}. Lecture Notes 165 in Applied and Computational Mechanics 73, Springer; 2014.

\bibitem{chev-2017} Cheviakov A. Symbolic computation of equivalence transformations and parameter reduction for nonlinear physical models. {\it Comp Phys Comm}.2017;220: 56--73.

\bibitem{clark} Clarkson PA, Mansfield EL, Priestley TJ. Symmetries of a class of nonlinear third-order partial differential equations. {\it Math Comput Modelling}. 1997;25: 195--212.

\bibitem{const} Constantin A. Existence of permanent and breaking waves for a shallow water equation: a geometric approach. {\it Ann. Inst. Fourier}. 2000;50: 321--362.

\bibitem{depro} Degasperis A, Procesi M. {\it Asymptotic integrability, in: Symmetry and Perturbation Theory}. World Scientific; 1999.

\bibitem{stelios1} Dimas S, Tsoubelis D. {\it SYM: A new symmetry-finding package for Mathematica}. Proceedings of the 10th International Conference in Modern Group Analysis, Larnaca, Cyprus. 2004:64--70.

\bibitem{stelios2} Dimas S, Tsoubelis D. {\it A new heuristic algorithm for solving overdetermined systems of PDEs in Mathematica}. 6th International Conference on Symmetry in Nonlinear Mathematical Physics, Kiev, Ukraine. 2005:20--26.

\bibitem{igor-cm}Freire IL. A look on some results about Camassa–Holm type equations. {\it Communications in Mathematics}. 2021;29: 115–130.

\bibitem{igor-jpa} Freire IL. Conserved quantities, continuation and compactly supported solutions of some shallow water models. {\it J. Phys. A: Math. Theor}. 2021;54: 015207.

\bibitem{freire-cor}Freire IL. Corrigendum: Conserved quantities, continuation and compactly supported solutions of some shallow water models. {\it J. Phys. A: Math. Theor}. 2021;54: 409502.

\bibitem{rodrigo} Freire IL, Tito RS. {\it A Novikov equation describing pseudo-spherical surfaces, its pseudo-potentials, and local isometric immersions}, Studies in Applied Mathematics, (2021), DOI: 10.1111/sapm.12457.


\bibitem{hone2} Hone ANW, Wang JP. Integrable peakon equations with cubic nonlinearity. {\it J Phys A: Math Theor}.2008;41:372002.

\bibitem{hydon} Hydon, PE. {\it Symmetry Methods for Differential Equations: A Beginner's Guide}, Cambridge University Press; 2000.

\bibitem{i1} Ibragimov NH. {\it Elementary Lie Group Analysis and Ordinary Differential Equations}, John Wiley and Sons; 1999.

\bibitem{li-na} Li J, Yin Z., Well-poseness and global existence for a generalized Degasperis--Procesi equation. {\it Nonlinear Anal RWA}. 2016;28:72--92.

\bibitem{linares} Linares F, Ponce G. Unique continuation properties for solutions to the Camassa--Holm equation and related models. {\it Proc. Amer. Math. Soc.} 2020;148: 3871--3879.

\bibitem{liu-jde} Mi Y, Liu Y, Guo B., Luo T. The Cauchy problem for a generalized Camassa-Holm equation. {\it J Diff. Equ}. 2019;266:6739--6770.


\bibitem{nov} Novikov V. Generalizations of the Camassa--Holm equation. J Phys A: Math Theor. 2009;42:342002. 

\bibitem{olverbook} Olver PJ. {\it Applications of Lie groups to differential equations}. Second ed. Springer; (1993).

\bibitem{reyes2011} Reyes EG. Equations of pseudo-spherical type (After S. S. Chern and K. Tenenblat). {\it Results Math}. 2011;60: 53--101, (2011).

\bibitem{taylor} Taylor ME. {\it Partial Differential Equations I}.Second ed. Springer; (2011).


\end{thebibliography}
\end{document}